\documentclass[12pt]{amsart}
\usepackage{amssymb,amscd,graphicx,latexsym,amsthm,amsmath,array,pdfsync}
\usepackage{hyperref}
\usepackage{baskervald}
\usepackage[all]{xy}

\newcolumntype{L}[1]{>{\raggedright\let\newline\\\arraybackslash\hspace{0pt}}m{#1}}
\newcolumntype{C}[1]{>{\centering\let\newline\\\arraybackslash\hspace{0pt}}m{#1}}
\newcolumntype{R}[1]{>{\raggedleft\let\newline\\\arraybackslash\hspace{0pt}}m{#1}}

\RequirePackage[dvipsnames,usenames]{xcolor}
\hypersetup{
bookmarks,
bookmarksdepth=3,
bookmarksopen,
bookmarksnumbered,
pdfstartview=FitH,
colorlinks,backref,hyperindex,
linkcolor=Sepia,
citecolor=BlueViolet,
}

\usepackage{tikz,amsthm}

\newtheorem{theorem}{Theorem}
\numberwithin{theorem}{section}
\newtheorem{lemma}[theorem]{Lemma}
\newtheorem{corollary}[theorem]{Corollary}
\newtheorem{proposition}[theorem]{Proposition}

\newtheorem{subtheorem}{Theorem}
\numberwithin{subtheorem}{subsection}
\newtheorem{sublemma}[subtheorem]{Lemma}

\newtheorem{subproposition}[subtheorem]{Proposition}

\theoremstyle{definition}
\newtheorem{subdefinition}[subtheorem]{Definition}

\theoremstyle{remark}
\newtheorem{subremark}[subtheorem]{Remark}
\theoremstyle{theorem}
\newtheorem*{Theorem B}{Theorem B}
\newtheorem{theoremalpha}{Theorem}

\theoremstyle{definition}
\newtheorem{definition}[theorem]{Definition}

\theoremstyle{remark}
\newtheorem{remark}[theorem]{Remark}

\numberwithin{equation}{section}

\newcommand{\enl}[2]{{#2}^{\lbrack #1 \rbrack}}
\newcommand{\hns}[2]{{#2}^{\lbrack #1 \rbrack}}
\newcommand{\slhf}[2]{\mu_{#1}(#2)}
\newcommand{\slgf}[2]{\mu^{#1}(#2)}
\newcommand{\br}[1]{\lbrack #1 \rbrack}
\newcommand{\sing}[1]{\mathrm{Sing}( #1 )}
\newcommand{\sym}[2]{\mathrm{Sym}^{#1}(#2)}
\newcommand{\Oc}{\mathcal{O}}
\newcommand{\Lc}{\mathcal{L}}
\newcommand{\Qc}{\mathcal{Q}}
\newcommand{\Fc}{\mathcal{F}}
\newcommand{\Ec}{\mathcal{E}}

\newcommand{\Mcal}{\mathcal{M}}
\newcommand{\cc}{\mathbb{C}}
\newcommand{\zz}{\mathbb{Z}}
\newcommand{\qq}{\mathbb{Q}}
\newcommand{\snf}{\text{$\mathfrak{S}_n$}}
\newcommand{\pt}{\mathbb{P}^2}
\newcommand{\ra}{\rightarrow}
\newcommand{\hsn}{H_{S^n}}

\newcommand{\dfn}{:=}

\newcommand{\rk}[1]{\mathrm{rank}(#1)}

\newcommand{\enlo}{\enl{n}{\Oc_{\pt}(1)}}
\newcommand{\hnpt}{\hns{n}{\pt}}

\newcommand{\hr}[2]{\hyperref[#1]{#2}}
\newcommand{\hhom}{\mathcal{H}\mathrm{om}}

\newcommand{\Zcal}{\mathcal{Z}}
\newcommand{\Zcaln}{{\mathcal{Z}_n}}
\newcommand{\pv}[1]{\mathbb{P}(#1)}
\newcommand{\tnl}{\enl{n}{(T_S)}}
\newcommand{\dern}{\mathrm{Der}_{\cc}(\mathrm{-log}B_n)}

\begin{document}

\title{Geometry and Stability of Tautological Bundles on Hilbert schemes of points}

\author{David Stapleton}

\maketitle


\section*{Introduction}

The purpose of this paper is to explore the geometry of tautological bundles on Hilbert schemes of smooth surfaces and to establish the slope stability of these bundles.

Let $S$ be a smooth complex projective surface, and denote by $\hns{n}{S}$ the Hilbert scheme parametrizing length $n$ subschemes of $S$. This parameter space carries some natural tautological vector bundles: if $\Lc$ is a line bundle on $S$ then $\enl{n}{\Lc}$ is the rank $n$ vector bundle whose fiber at the point corresponding to a length $n$ subscheme $\xi \subset S$ is the vector space $H^0(S,\Lc \otimes \Oc_\xi)$. These tautological vector bundles have attracted a great deal of interest. Danila ~\cite{Danila} and Scala ~\cite{Scala} computed their cohomology. Ellingsrud and Str\o mme ~\cite{EStromme} showed the Chern classes of the bundles $\enl{n}{\Oc_{\mathbb{P}^2}}$, $\enl{n}{\Oc_{\mathbb{P}^2}(1)}$, and $\enl{n}{\Oc_{\mathbb{P}^2}(2)}$ generate the cohomology of $\hns{n}{\mathbb{P}^2}$. Nakajima gave an interpretation of the McKay correspondence by restricting the tautological bundles to the G-Hilbert scheme which is nicely exposited in ~\cite[$\S$4.3]{NakHilb}. Recently Okounkov ~\cite{Okounkov} formulated a conjecture about special generating functions associated to the tautological bundles.

Given the importance of the tautological bundles it is natural to ask whether they are stable. In ~\cite{Schl}, ~\cite{Wandel1}, and ~\cite{Wandel2} this question has been answered positively for Hilbert schemes of 2 points or 3 points on a K3 or abelian surface with Picard group restrictions. Our first result establishes the stability of these bundles for arbitrary $n$ and any surface.


\begin{theoremalpha}\label{A} If $\Lc$ is a nontrivial line bundle on $S$, then $\enl{n}{\Lc}$ is slope stable with respect to natural Chow divisors on $\hns{n}{S}$.
\end{theoremalpha}

\noindent More precisely, an ample divisor on $S$ determines a natural ample divisor on $\sym{n}{S}$, and the pullback via the Hilbert-Chow morphism gives one such natural Chow divisor on $\hns{n}{S}$, which is not ample but is big and semiample. More generally, we prove that if $\Ec \not\cong \Oc_S$ is any slope stable vector bundle on $S$ with respect to some ample divisor then $\enl{n}{\Ec}$ is slope stable with respect to the corresponding Chow divisor. Although Theorem A only gives stability with respect to a strictly big and nef divisor, we are able to deduce stability with respect to nearby ample divisors via a perturbation argument on the nef cone.

If $S$ is any smooth surface, there is a divisor $B_n$ in $\hns{n}{S}$ which consists of nonreduced subschemes. The pair $(\hns{n}{S},B_n)$ gives a natural closure of the space of $n$ distinct points in $S$. The vector fields on $\hns{n}{S}$ tangent to $B_n$ form the sheaf of logarithmic vector fields $\dern$. Our second result says the sheaf $\dern$ is naturally isomorphic to the tautological bundle associated to the tangent bundle on $S$.


\begin{theoremalpha}\label{B} For any smooth surface $S$ there exists a natural injection:
\begin{center}
$\displaystyle \alpha_n : \tnl \ra T_{\hns{n}{S}}$,
\end{center}
\noindent and $\alpha_n$ induces an isomorphism between $\tnl$ and $\dern$.
\end{theoremalpha}

\noindent The analogous statement also holds for smooth curves. In general the sheaves $\dern$ are only guaranteed to be reflexive as $B_n$ is not simple normal crossing. However, \hr{B}{Theorem B} shows the $\dern$ is locally free, that is $B_n$ is a \textit{free divisor}. Buchweitz and Mond were already aware that $B_n$ is a free divisor as is indicated in the introduction of \cite{Buch}.

Finally, we explore the geometry of the tautological bundles when the surface is the projective plane. We prove the tautological bundles on $\hnpt$ are rich enough to capture all semistable rank $n$ bundles on curves.


\begin{theoremalpha}\label{C} If $C$ is a smooth projective curve and $\Ec$ is a semistable rank $n$ vector bundle on $C$ with sufficiently positive degree, then there exists an embedding $C \ra \hns{n}{\mathbb{P}^2}$ such that $\enl{n}{\Oc_{\mathbb{P}^2}(1)}|_C \cong  \Ec$.
\end{theoremalpha}

The proof of \hr{A}{Theorem A} follows the approach taken by Mistretta ~\cite{Mistretta} who studies the stability of tautological bundles on the symmetric powers of a curve. The idea is to examine the tautological vector bundles on the cartesian power $S^n$ and show there are no $\snf$-equivariant destabilizing subsheaves. This strategy is more effective for surfaces because the diagonals in $S^n$ have codimension 2. The map in \hr{B}{Theorem B} arises from pushing forward the normal sequence of the universal family. The image of the restriction of the map to a point $\br{\xi} \in \hns{n}{S}$ can be thought of as deformations of $\xi$ coming from vector fields on $S$. The proof of \hr{C}{Theorem C} is constructive, using the spectral curves of \cite{BNR}.

In \hr{Sec1}{Section 1} we give the proof of \hr{A}{Theorem A}. In \hr{Sec2}{Section 2} we explore the geometry of the tautological bundles. We start by proving \hr{C}{Theorem C}. We proceed by showing that for Hilbert schemes of 2 points in the plane the tautological bundles are relative analogues of the Steiner bundles in the plane. Next we prove \hr{B}{Theorem B}. Finally, in \hr{Sec3}{Section 3} we give the perturbation argument, deducing the tautological bundles are stable with respect to ample divisors.

Throughout we work over the complex numbers. For any divisor class $H \in N^1(X)$. We define the \textit{slope of $\Ec$ with respect to $H$} to be the rational number:

\begin{center}
$\displaystyle \slhf{H}{\Ec} \dfn \dfrac{c_1(\Ec) \cdot H^{d-1}}{ \rk{\Ec}}$.
\end{center}

\noindent We say $\Ec$ is \textit{slope (semi)stable with respect to H} if for all subsheaves $\Fc \subset \Ec$ of intermediate rank:

\begin{center}
$\slhf{H}{\Fc} \underset{(\le)}{<} \slhf{H}{\Ec}$.
\end{center}

I am grateful to my advisor Robert Lazarsfeld who suggested the project and directed me in productive lines of thought. I am also thankful for conversations and correspondences with Lawrence Ein, Roman Gayduk, Daniel Greb, Giulia Sacc\`a, Ian Shipman, Brooke Ullery, Dingxin Zhang, and Xin Zhang. This paper is a substantial revision of a previous preprint. I would finally like to thank the referee of the previous paper for a thorough review and helpful suggestions.


\section{Stability of Tautological Bundles}\label{Sec1}

In this section we prove that the tautological bundle of a stable vector bundle $\Ec$ is stable with respect to natural Chow divisors on $\hns{n}{S}$. Thus we deduce \hr{A}{Theorem A} when $\Ec$ is a nontrivial line bundle. We start by defining the essential objects in the study of Hilbert schemes of points on surfaces.

Let $S$ be a smooth complex projective surface. We write $\hns{n}{S}$ for the Hilbert scheme of length $n$ subschemes of $S$. We denote by $\mathcal{Z}_n$ the universal family of $\hns{n}{S}$ with projections:

\begin{center}
\begin{tikzpicture}
  \node (Zn) {$S \times \hns{n}{S} \supset \mathcal{Z}_n$};
  \node (S) [right] at (3,0) {$S$.};
  \node (HnS) [below] at (.95,-1) {$\hns{n}{S}$};
  \node (p1) at (2.1,.2) {$p_1$};
  \node (p2) at (.75,-.65) {$p_2$};
  \draw[->] (.95,-.3) -- (.95,-1.05);
  \draw[->] (1.2,0) -- (3,0);
\end{tikzpicture}
\end{center}

\noindent For a fixed vector bundle $\Ec$ on $S$ of rank $r$ we define

\begin{center}
$\enl{n}{\Ec} \dfn {p_2}_*({p_1}^*\Ec)$.
\end{center}

\noindent It is the \textit{tautological vector bundle associated to $\Ec$} and has rank $rn$. The fiber of $\enl{n}{\Ec}$ at a point $\br{\xi} \in \hns{n}{S}$ can be naturally identified with the vector space $H^0(S,\Ec|_\xi)$.

The symmetric group on $n$ elements $\snf$ naturally acts on the cartesian product $S^n$, and we write $\sigma_n$ for the quotient map:

\begin{center}
$\sigma_n:S^n \ra S^n / \snf =: \sym{n}{S}$.
\end{center}

\noindent There is also a Hilbert-Chow morphism:

\begin{center}
$h_n : \hns{n}{S} \ra \sym{n}{S}$
\end{center}

\noindent which is a semismall map ~\cite[Definition 2.1.1]{dCM1}.

We wish to view $\enl{n}{\Ec}$ as an $\snf$-equivariant sheaf on $S^n$. Recall that if $G$ is a finite group that acts on a scheme $X$, and if $\Fc$ is a coherent sheaf on $X$ then a \textit{$G$-equivariant structure on $\Fc$} is given by a choice of isomorphisms:

\begin{center}
$\phi_g : \Fc \ra g^* \Fc$
\end{center}

\noindent for all $g \in G$ satisfying the compatibility condition $h^* (\phi_g) \circ \phi_h = \phi_{gh}$. Following Danila ~\cite{Danila} and Scala ~\cite{Scala} we study the tautological bundles on $\hns{n}{S}$ by working with $\snf$-equivariant sheaves on $S^n$. For our purposes it is enough to study $\enl{n}{\Ec}$ equivariantly on the open subset of distinct points in $\hns{n}{S}$.

We write $\sym{n}{S}_{\circ}$ for the open subset of $\sym{n}{S}$ of distinct points. Likewise given a map $f : X \ra \sym{n}{S}$ we write $X_{\circ}$ for $f^{-1}(\sym{n}{S}_{\circ})$. By abuse of notation given another map $g: X \ra Y$ with domain $X$ we define $g_{\circ} \dfn g|_{X_{\circ}}$ and given a coherent sheaf $\Fc$ on $X$ we define $\Fc_{\circ} \dfn \Fc|_{X_{\circ}}$. The map $h_{n,\circ} : \hns{n}{S}_{\circ} \ra \sym{n}{S}_{\circ}$ is an isomorphism. We define
\begin{center}
$\overline{\sigma}_{n,\circ} \dfn h_{n,\circ}^{-1} \circ \sigma_{n,\circ}:S^n_{\circ} \ra \hns{n}{S}_{\circ}$.
\end{center}
Given a torsion-free coherent sheaf $\Fc$ on $\hns{n}{S}$ we define a torsion-free coherent sheaf on $S^n$ by

\begin{center}
$(\Fc)_{S^n} \dfn j_*(\overline{\sigma}_{n,\circ}^*(\Fc_{\circ}))$
\end{center}

\noindent where $j$ is the inclusion $j:S^n_{\circ} \ra S^n$. The sheaf $(\Fc)_{S^n}$ can be thought of as a modification of $\Fc$ along the exceptional divisor of $h_n$. 

The pullback $\overline{\sigma}_{n,\circ}^* (-)$ is left exact as the map $\overline{\sigma}_{n,\circ}$ is \'etale; thus the functor $(-)_{S^n}$ is left exact. If $\Fc$ is reflexive, the normality of $S^n$ implies the natural $\snf$-equivariant structure on the reflexive sheaf $\overline{\sigma}_{n,\circ}^*(\Fc_{\circ})$ pushes forward uniquely to an $\snf$-equivariant structure on $(\Fc)_{S^n}$.

Let $q_i$ denote the projection from $S^n$ onto the $i$th factor. Given a vector bundle $\Ec$ on $S$ there is an $\snf$-equivariant vector bundle on $S^n$ defined by

\begin{center}
$\displaystyle \Ec^{\boxplus n} \dfn \overset{n}{\underset{i=1}\bigoplus} q_i^*(\Ec)$.
\end{center}

\noindent We have given two natural $\snf$-equivariant sheaves on $S^n$ associated to $\Ec$. In fact they are equivalent.


\begin{lemma}\label{1.1} Given a vector bundle $\Ec$ on $S$ there is an isomorphism:
\begin{center}
$(\enl{n}{\Ec})_{S^n} \cong \Ec^{\boxplus n}$
\end{center}
\noindent of $\snf$-equivariant vector bundles on $S^n$.
\end{lemma}

\begin{proof} Consider the fiber square:

\begin{center}
\begin{tikzpicture}
  \node (Fib) at (2,0) {$\displaystyle \mathcal{Z}_{n,\circ} \underset{\text{\tiny ${\hns{n}{S}_{\circ}}$}}{\times} S^n_{\circ}$};
  \node (F) at (.45,0) {$F \dfn$};
  \node (Sn) at (5,0) {$S^n_{\circ}$};
  \node (Zn) at (2,-2) {$\displaystyle \mathcal{Z}_{n,\circ}$};
  \node (Hn) at (5,-2) {$\displaystyle \hns{n}{S}_{\circ}$};
  \draw[->] (Fib) to node[above] {$p'_{2,o}$} (Sn);
  \draw[->] (Fib) to node[left] {$\overline{\sigma}'_{n,\circ}$} (Zn);
  \draw[->] (Zn) to node[above] {$p_{2,o}$} (Hn);
  \draw[->] (Sn) to node[right] {$\overline{\sigma}_{n,\circ}$} (Hn);
\end{tikzpicture}.
\end{center}

\noindent Every map in the fiber square is an \'etale map between $\snf$-schemes (the $\snf$-action on $\mathcal{Z}_{n,\circ}$ and $\hns{n}{S}_{\circ}$ is trivial). We write $\Gamma_i$ for the subscheme of $S^n_{\circ} \times S$ that is the graph of the map $q_{i,o} : S^n_{\circ} \ra S$. The scheme $F$ is equal to the disjoint union $\coprod \Gamma_i$ and is a subscheme of $S^n_{\circ} \times S$. The restriction $p_{1,\circ} \circ \overline{\sigma}'_{n,\circ}|_{\Gamma_i}$ is the projection $\Gamma_i \ra S$. So there is an equivariant isomorphism ${p'_{2,o}}_*({\overline{\sigma}'_{n,\circ}}^*({p_{1,\circ}}^*(\Ec))) \cong \Ec^{\boxplus n}_{\circ}$.

As the fiber square is made of flat proper $\snf$-maps there is a natural $\snf$-equivariant isomorphism:

\begin{center}
${p'_{2,o}}_*({\overline{\sigma}'_{n,\circ}}^*({p_{1,\circ}}^*(\Ec))) \cong {\overline{\sigma}_{n,\circ}}^*({p_{2,o}}_*({p_{1,\circ}}^*(\Ec)))$.
\end{center}

\noindent The latter sheaf is $(\enl{n}{\Ec})_{S^n,\circ}$. Finally, any isomorphism between vector bundles on $S^n_{\circ}$ uniquely extends to an isomorphism between their pushforwards along $j$. Therefore there is a natural $\snf$-equivariant isomorphism $(\enl{n}{\Ec})_{S^n} \cong \Ec^{\boxplus n}$.
\end{proof}

Given an ample divisor $H$ on $S$ there is a natural $\snf$-invariant ample divisor on $S^n$ defined as:

\begin{center}
$\hsn:=\overset{n}{\underset{i=1}\sum} q_i^*(H)$.
\end{center}

\noindent Fogarty ~\cite[Lemma 6.1]{Fogarty} shows every divisor $\hsn$ descends to an ample Cartier divisor on $\sym{n}{S}$. Pulling back this Cartier divisor along the Hilbert-Chow morphism gives a big and nef divisor on $\hns{n}{S}$ which we denote by $H_n$. If $H$ is effective then $H_n$ can be realized set-theoretically as

\begin{center}
$H_n = \{ \xi \in \hns{n}{S} \text{ }|\text{ } \xi \cap \mathrm{Supp}(H) \ne \emptyset \}$.
\end{center}


\begin{lemma}\label{1.2} If $\Fc$ is a torsion-free sheaf on $\hns{n}{S}$ then

\begin{center}
$\displaystyle (n!)\int\limits_{\hns{n}{S}} c_1(\Fc) \cdot (H_n)^{2n-1} = \int\limits_{S^n}c_1((\Fc)_{S^n}) \cdot (H_{S^n})^{2n-1}$.
\end{center}
\end{lemma}

\begin{proof} This is a straightforward calculation using $\hns{n}{S}_{\circ}$, $\sym{n}{S}_{\circ}$, and $S^n_{\circ}$.
\end{proof}


In the following lemma we assume \hr{3.7}{Proposition 3.7} which says the pullback of a stable bundle to a product is stable with respect to a product polarization. For the sake of the exposition we give the proof of \hr{3.7}{Proposition 3.7} in \hr{Sec3}{Section 3}. 

\begin{lemma}\label{1.3} If $\Ec \not\cong \Oc_S$ is slope stable on $S$ with respect to an ample divisor $H$ then there are no $\snf$-equivariant subsheaves of $\Ec^{\boxplus n}$ that are slope destabilizing with respect to 
$H_{S^n}$.
\end{lemma}

\begin{proof} Let $0 \ne \Fc \subset \Ec^{\boxplus n}$ be an $\snf$-equivariant subsheaf. We can find a (not necessarily equivariant) slope stable subsheaf $0 \ne \Fc' \subset \Fc$ which has maximal slope with respect to $H_{S^n}$. Fix $i$ so that the composition:

\begin{center}
$\Fc' \ra \Ec^{\boxplus n} \ra q_i^* \Ec$
\end{center}

\noindent is nonzero. By \hr{3.7}{Proposition 3.7} we know that each $q_i^* \Ec$ is slope stable with respect to $H_{S^n}$. A nonzero map between slope stable sheaves can only exist if

\begin{enumerate}
\item the slope of $\Fc'$ is less than the slope of $q_i^* \Ec$, or
\item $\Fc' \ra q_i^* \Ec$ is an isomorphism.
\end{enumerate}

In case (1), $\displaystyle \slhf{H_{S^n}}{\Fc} \le \slhf{H_{S^n}}{\Fc'} < \slhf{H_{S^n}}{q_i^* \Ec}$. By symmetry, $\slhf{H_{S^n}}{q_i^* \Ec} = \slhf{H_{S^n}}{q_j^* \Ec}$ for all $i$ and $j$. Thus $\slhf{H_{S^n}}{q_i^* \Ec} = \slhf{H_{S^n}}{\Ec^{\boxplus n}}$ and $\Fc$ does not destabilize $\Ec^{\boxplus n}$.

In case (2), we know $\Fc' \cong q_i^* \Ec$. Because $\Ec \not\cong \Oc_S$, the pullbacks $q_i^* \Ec$ and $q_j^* \Ec$ are not isomorphic unless $i = j$. As all the $q_j^* \Ec$ have the same slope and are stable with respect to $H_{S^n}$, $\mathrm{Hom}(\Fc',q_j^* \Ec)=0$ for $j \ne i$. In particular all the compositions

\begin{center}
$\Fc' \ra \Ec^{\boxplus n} \ra q_j^* \Ec$
\end{center}

\noindent are zero for $j \ne i$. Thus $\Fc'$ is a summand of $\Ec^{\boxplus n}$. So $\Fc$ is an $\snf$-equivariant subsheaf of $\Ec^{\boxplus n}$ which contains one of the summands. But $\snf$ acts transitively on the summands so $\Fc$ contains all the summands, hence $\Fc$ does not destabilize $\Ec^{\boxplus n}$.
\end{proof}

Now we prove \hr{A}{Theorem A} in full generality.

\begin{theorem}\label{Agen} If $\Ec \not\cong \Oc_S$ is a vector bundle on $S$ which is slope stable with respect to an ample divisor $H$, then $\enl{n}{\Ec}$ is slope stable with respect to $H_n$.
\end{theorem}


\begin{proof}\label{pA} Let $\Fc \subset \enl{n}{\Ec}$ be a reflexive subsheaf of intermediate rank. It is enough to consider reflexive sheaves because the saturation of a torsion free subsheaf of $\enl{n}{\Ec}$ is reflexive of the same rank and its slope cannot decrease. By \hr{1.2}{Lemma 1.2}, the slope of a torsion-free sheaf $\Fc$ with respect to $H_n$ is up to a fixed positive multiple the same as the slope of $(\Fc)_{S^n}$ with respect to $H_{S^n}$. In particular

\begin{center}
$\slhf{H_n}{\Fc}<\slhf{H_n}{\enl{n}{\Ec}} \iff \slhf{H_{S^n}}{(\Fc)_{S^n}} < \slhf{H_{S^n}}{\Ec^{\boxplus n}}$.
\end{center}

\noindent Now $(\Fc)_{S^n}$ is naturally an $\snf$-equivariant subsheaf of $\Ec^{\boxplus n}$. Thus by \hr{1.3}{Lemma 1.3}

\begin{center}
$\slhf{H_{S^n}}{(\Fc)_{S^n}} < \slhf{H_{S^n}}{\Ec^{\boxplus n}}$.
\end{center}

\noindent Therefore, $\slhf{H_n}{\Fc}<\slhf{H_n}{\enl{n}{\Ec}}$ for all torsion-free subsheaves of intermediate rank, and $\enl{n}{\Ec}$ is stable with respect to $H_n$.
\end{proof}

\begin{remark}[On the Bogomolov inequalities]
If $\Ec$ is a vector bundle on $S$ stable with respect to $H$, then stability of $\enl{n}{\Ec}$ with respect to $H_n$ along with the perturbation argument from Section 3 implies the Bogomolov type topological inequality:
\begin{center}
$(r-1)c_1(\enl{n}{\Ec})^2 \cdot H_n^{2n-2} \le 2r c_2(\enl{n}{\Ec})\cdot H_n^{2n-2}$.
\end{center}
\noindent These intersection numbers can be rewritten in terms of the intersection theory of $S$ and these inequalities reduce to the regular Bogomolov inequality on $S$ and the inequality coming from the Hodge index theorem.
\end{remark}


\section{Geometry of tautological bundles}\label{Sec2}

In this section we give examples of the geometry inherent to the tautological vector bundles. For each curve we construct an embedding in the Hilbert scheme of $n$ points in the plane such that the restriction of $\enlo$ to the curve is some prescribed vector bundle. Next, we show that the tautological bundles on the space of 2 points in the plane have explicit resolutions, making them a relative analogue of Steiner bundles in the plane. Finally, for any smooth surface we construct a map from the tautological bundle of the tangent bundle of the surface to the tangent bundle of the Hilbert scheme of points on the surface which realizes the first bundle as the sheaf of logarithmic vector fields.


\subsection{Restrictions to curves}
In this section we prove every sufficiently positive, rank $n$, semistable vector bundle on a smooth projective curve arises as the pull back of $\enlo$ along an embedding of the curve in $\hnpt$. To prove the theorem we need the spectral curves of \cite{BNR}. For completeness we recall the construction.

Let $\pi : D \ra C$ be an $n:1$ map between smooth irreducible projective curves and let $\Ec$ be an $\Oc_C$-module. If $D$ can be embedded into the total space of a line bundle $\Lc$ on $C$:

\begin{center}
$\mathbb{L} \dfn \mathcal{S}pec_{\Oc_C} (\sym{\bullet}{\Lc^{\vee}}) \xrightarrow{\pi_{\mathbb{L}}} C$
\end{center}

\noindent with $\pi = \pi_{\mathbb{L}}|_D$ then this gives a presentation:

\begin{center}
$\pi_* \Oc_D \cong \sym{\bullet}{\Lc^{\vee}} \Big/ (x^n + s_1 x^{n-1} + ... + s_n)$
\end{center}

\noindent for $x^n + s_1 x^{n-1} + ... + s_n \in H^0(\mathbb{L} , ({\pi_{\mathbb{L}}}^* \Lc)^{\otimes n})$. Here we write $x \in H^0(\mathbb{L},{\pi_{\mathbb{L}}}^*(\Lc))$ for the \textit{coordinate section} of ${\pi_{\mathbb{L}}}^*(\Lc)$. To give $\Ec$ the structure of a $\pi_*\Oc_D$-module we need to specify a multiplication map $m: \Ec \otimes \Lc^{-1} \ra \Ec$ (equivalently $\Ec \ra \Ec \otimes \Lc$) which satisfies the relation
$m^n + s_1 m^{n-1} + ... + s_n = 0$.

Every $\Lc$-twisted endomorphism $m : \Ec \ra \Ec \otimes \Lc$ has an associated $\Lc$-twisted characteristic polynomial, which is a global section $p_m(x) \in H^0(\mathbb{L},({\pi_{\mathbb{L}}}^* \Lc)^{\otimes n})$. A global version of the Cayley-Hamilton theorem says that $m$ automatically satisfies its $\Lc$-twisted characteristic polynomial. In particular, if the zero set of $p_m(x)$ is $D$ then $\Ec$ can naturally be thought of as a $\pi_*\Oc_D$-module. Fixing $s \in H^0(\mathbb{L},({\pi_{\mathbb{L}}}^* \Lc)^{\otimes n})$ which cuts out the integral curve $D$, \cite[Proposition 3.6]{BNR} gives the beautiful correspondence:
\begin{equation}\label{eqn:diamond}
\left\{ \mathcal{E} \xrightarrow{m} \mathcal{E} \otimes \Lc \text{ } \Big| \mathcal{E} \text{ a vector bundle and } p_m(x) = s \right\} \stackrel{1:1}{\longleftrightarrow}  \{ \text{invertible sheaves } \Mcal \text{ on } D \}.\tag{$\diamond$}
\end{equation}
The correspondence going from right to left is given by taking the coordinate section of ${\pi_{\mathbb{L}}}^*(\Lc)$, restricting to $D$, twisting by $\Mcal$, and pushing forward along $\pi$.


\begin{proof}[Proof of \hr{C}{Theorem C}] Let $C$ be a smooth projective genus $g$ curve and $\Ec$ a rank $n$ semistable vector bundle on $C$. Let $\Lc$ be a line bundle on $C$. As $\Ec$ is semistable, $\Ec \otimes \Ec^{\vee}$ is also semistable and has slope 0, hence $\Ec \otimes \Ec^{\vee} \otimes \Lc$ is globally generated when $\deg{\Lc} \ge 2g$. Therefore, if

\begin{center}
$m: \Ec \ra \Ec \otimes \Lc$
\end{center}

\noindent is a general $\Lc$-twisted endomorphism then the resulting $\Lc$-twisted characteristic polynomial is smooth with simple branching. In fact, if $V \subset \mathbb{E} \otimes \mathbb{E}^{\vee} \otimes \mathbb{L}$ is the locus of $\Lc$-twisted endomorphisms whose characteristic polynomial has repeated roots and if $m : C \ra \mathbb{E} \otimes \mathbb{E}^{\vee} \otimes \mathbb{L}$ meets $V$ transversely and avoids the locus of $V$ where the $\Lc$-twisted characteristic polynomial has repeated roots to a higher multiplicity, then the resulting spectral curve $D$ is smooth and connected.

Thus, by the correspondence \eqref{eqn:diamond} there is a line bundle $\Mcal$ on $D$ such that $\pi_*\Mcal \cong \Ec$. The genus of $D$ is $g_D = {r \choose 2} \mathrm{deg}(\Lc) + n(g-1) + 1$ and is independent of $\Ec$. However, the degree of $\Mcal$ is $ \mathrm{deg}(\Ec) + {r \choose 2}\mathrm{deg}(\Lc)$ and does depend on the degree of $\Ec$. In particular, if

\begin{center}
$\mathrm{deg}(\Ec) \ge {r \choose 2}\mathrm{deg}(\Lc) + r(2g-2) + 3$
\end{center}

\noindent then $\Mcal$ is very ample and 3 general sections of $\Mcal$ give a map $\phi : D \ra \pt$ such that the induced maps $\pi \times \phi : D \ra C \times \pt$ and $\psi_{\pi,\phi} : C \ra \hnpt$ are embeddings. Under the embedding $\psi_{\pi,\phi}$ the restriction of $\enlo$ to $C$ is precisely $\Ec$, proving Theorem C.
\end{proof}


\subsection{Two points in the projective plane}
Now we restrict our attention to the Hilbert scheme of 2 points in the plane. As a reminder, if we identify $\pt = \sym{2}{\mathbb{P}^1}$ (where $\mathbb{P}^1 = \mathbb{P}(W)$), then for $k\ge 2$ the tautological bundle $\enl{2}{\Oc_{\mathbb{P}^1}(k)}$ has a natural 2-term resolution:

\begin{center}
$\displaystyle 0 \ra \sym{k-2}{W} \otimes_\cc \Oc_{\pt}(-1) \xrightarrow{m} \sym{k}{W} \otimes_\cc \Oc_{\pt} \ra \enl{2}{\Oc_{\mathbb{P}^1}(k)} \ra 0$.
\end{center}

\noindent Here

\begin{center}
$\displaystyle m \in \mathrm{Hom}(\sym{k-2}{W} \otimes_\cc \Oc_{\pt}(-1),\sym{k}{W} \otimes_\cc \Oc_{\pt})$\\
\hspace{4cm}$\cong \mathrm{Hom}(\sym{k-2}{W} \otimes \sym{2}{W},\sym{k}{W})$
\end{center}

\noindent is the multiplication in the symmetric algebra. In general, any bundle on $\mathbb{P}^N$ with a similar 2-term linear resolution is called a Steiner bundle.

The Hilbert scheme of two points on the projective plane $\pv{V}$ has a natural map to $\pv{V^\vee}$. The map
\begin{center}
$\psi : \hns{2}{\pv{V}} \ra \pv{V^\vee}$
\end{center}
\noindent exhibits $\hns{2}{\pv{V}}$ as a two dimensional projective bundle over $\pv{V^\vee}$. Geometrically, $\psi$ is given by sending a subscheme $\br{\xi}$ to the line that is spanned by $\xi$. A fiber of $\psi$ is the second symmetric power of the corresponding line. Viewing $\hns{2}{\pv{V}}$ as a $\pt$-bundle over $\pv{V^\vee}$ the tautological bundles come with a two term \textit{relative linear} resolution, making them a relative version of the Steiner bundles in the plane.

\begin{subproposition}
For $k \ge 2$ there is a 2 term relatively linear resolution of $\enl{2}{\Oc_{\pv{V}}(k)}$ given by
\begin{center}
$\displaystyle 0 \ra \psi^*(\mathrm{Sym}^{k-2} K_{(V^\vee)}^\vee)(-1) \ra \psi^*(\mathrm{Sym}^{k} K_{(V^\vee)}^\vee) \ra \enl{2}{\Oc_{\pv{V}}(k)} \ra 0 $
\end{center}
\noindent where $K_{(V^\vee)}$ is the kernel bundle in the tautological sequence on $\pv{V^\vee}$:
\begin{center}
$0 \ra K_{(V^\vee)} \ra V^\vee \otimes \Oc_{\pv{V^\vee}} \ra \Oc_{\pv{V^\vee}}(1) \ra 0$.
\end{center}
\end{subproposition}

\begin{proof}[Sketch of proof] There is an isomorphism $\hns{2}{\pv{V}} \cong \mathbb{P}(\mathrm{Sym}^2(K_{(V^\vee)}))$. From this perspective, the universal family $\Zcal_2$ is a divisor in the fiber product $X \dfn \hns{2}{\pv{V}} \times_{\pv{V^\vee}} \mathbb{P}(K_{(V^\vee)}^\vee)$. Specifically $\Oc_X(\Zcal_2) \cong \Oc_X(1,2)$. And if $p:X \ra \hns{2}{\pv{V}}$ is the projection map then

\begin{center}
$p_* \big( \Oc_{\Zcal_2}(0,k)\big) \cong \enl{2}{\Oc_{\pv{V}}(k)}$.
\end{center}

\noindent Therefore, we can twist the ideal sequence of $\Zcal_2$ and take direct images of the sequence with respect to $p$ to obtain resolutions of the tautological bundles. When $k \ge 2$ there are no higher direct images, so the pushforward of the twisted ideal sequence is exact, giving the desired resolution.
\end{proof}

\begin{subremark} We can modify the proof to obtain resolutions of $\enl{2}{\Oc_{\pv{V}}(k)}$ for all $k$.
\begin{center}
\begin{tabular}{r|l}
$k=1$ & $\displaystyle \enl{2}{\Oc_{\pv{V}}(1)} \cong \psi^* (K_{(V^\vee)}^\vee) $\\

$k=0$ & $\displaystyle 0 \ra \Oc_{\hns{2}{\pv{V}}} \ra \enl{2}{\big( \Oc_{\pv{V}} \big)} \ra \psi^* \big(\mathrm{det} (K_{(V^\vee)}^\vee)\big)(-1) \ra 0$\\

$k=-1$ & $\displaystyle \enl{2}{\Oc_{\pv{V}}(-1)} \cong \psi^* (K_{V^\vee}) \otimes \psi^*\big(\Oc_{\pv{V^\vee}}(1) \big)(-1) $\\

$k\le -2$ & $\displaystyle 0 \ra \enl{2}{\Oc_{\pv{V}}(k)} \ra  \psi^*\Big( \big(\mathrm{Sym}^{k} K_{(V^\vee)}\big)(1)\Big)(-1) \ra \psi^*\Big( \big(\mathrm{Sym}^{k-2} K_{(V^\vee)}\big)(1) \Big) \ra 0$.\\
\end{tabular}
\end{center}
\end{subremark}

\begin{subremark}
For $N > 2$, the Hilbert scheme of 2 points on $\mathbb{P}^N$ is smooth. There is an analogous map:
\begin{center}$\psi : \hns{2}{(\mathbb{P}^N)} \ra \mathrm{Gr}(2,N+1)$,
\end{center}
\noindent and the tautological bundles have 2 term relative linear resolutions as in the case $N=2$.
\end{subremark}


\subsection{The tautological tangent map} For any smooth surface $S$ (not necessarily projective), the Hilbert scheme $\hns{n}{S}$ is a smooth closure of the space of $n$ distinct points in $S$. The boundary $B_n$ is the locus of nonreduced length $n$ subschemes of $S$. We are interested in vector fields which are tangent to the boundary $B_n$.

\begin{subdefinition}
If $D$ is a codimension 1 subvariety of $X$ a smooth variety, then the sheaf of logarithmic vector fields, denoted $\mathrm{Der}_\cc(\mathrm{-log}D)$, is the subsheaf of $T_X$ consisting of vector fields which along the regular locus of $D$ are tangent to $D$.
\end{subdefinition}

\noindent When $D$ is smooth, $\mathrm{Der}_\cc(\mathrm{-log}D)$ is just the elementary transformation of the tangent bundle along the normal bundle of $D$ in $X$, in particular it is a vector bundle. Even when $D$ is singular $\mathrm{Der}_\cc(\mathrm{-log}D)$ is reflexive, so it is enough to define $\mathrm{Der}_\cc(\mathrm{-log}D)$ away from the singular locus (or any codimension 2 set in $X$) of $D$ and then pushforward.

For Hilbert schemes of points on surfaces we can naturally understand $\mathrm{Der}_\cc(\mathrm{-log}B_n)$ as the tautological bundles of the tangent bundle on the surface.

\begin{Theorem B} For any smooth connected surface $S$ there exists a natural injection:
\begin{center}
$\displaystyle \alpha_n : \tnl \ra T_{\hns{n}{S}}$,
\end{center}
\noindent and $\alpha_n$ induces an isomorphism between $\tnl$ and $\dern$.
\end{Theorem B}

\noindent At a point $\br{\xi} \in \hns{n}{S}$ the map $\alpha_n|_{\br{\xi}}$ can be interpreted as deformations of $\xi$ coming from tangent vectors of $S$. We expect that the degeneracy loci of $\alpha_n$ give a interesting stratification of $\hns{n}{S}$.

Before proving \hr{B}{Theorem B} we prove a general lemma.

\begin{sublemma}
Let $X$ and $Y$ be smooth varieties and $f: X \ra Y$ a branched covering with reduced branch locus $B \subset Y$. If $\delta \in H^0(Y,TY)$ is a vector field on $Y$ whose pullback $f^* \delta \in H^0(X,f^*TY)$ is in the image of
\begin{center}
$df: H^0(X,TX) \ra H^0(X,f^*TY)$,
\end{center}
\noindent then $\delta \in H^0(Y,\mathrm{Der}_\cc(\mathrm{-log}B))$.
\end{sublemma}

\begin{proof} It is enough to check $\delta$ is tangent to $B$ for points $p \in B$ outside of a codimension 2 subset in $Y$. Let $p \in B$ be a general point and $q$ a ramified point in the fiber of $f$ over $p$. We can choose local analytic coordinates $y_1 , ... , y_n$ centered at $p$ and coordinates $x_1 , ... , x_n$ centered at $q$ such that

\begin{center}
$f^*(y_1)=x_1^m$\\ \hspace{1.3cm}$f^*(x_i)=x_i$ $(i>1)$.
\end{center}

\noindent That is $y_1$ is a local equation for $B$ and $x_1$ is a local equation for the reduced component of ramification containing $q$. Then the derivative $df$ maps

\begin{center}
$\frac{\partial}{\partial x_1} \mapsto m x_1^{m-1} f^* \big( \frac{\partial}{\partial y_1} \big)$\\
\hspace{.2cm}$\frac{\partial}{\partial x_i} \mapsto f^* \big( \frac{\partial}{\partial y_i} \big)$ $(i >1)$.
\end{center}

\noindent Now $f^* \delta$ is in the image of $df$. Expanding locally, $f^*\delta = f^*(g_1) f^* \big( \frac{\partial}{\partial y_1} \big) + ... + f^*(g_n) f^* \big( \frac{\partial}{\partial y_n} \big)$. Thus $x_1^{m-1}$ divides $f^* (g_1)$. So $y_1$ divides $g_1$ and $\delta$ is in $H^0(Y,\mathrm{Der}_\cc(\mathrm{-log}B))$.
\end{proof}

\begin{proof}[Proof of \hr{B}{Theorem B}] As in \S1 we use $\Zcaln \subset S \times \hns{n}{S}$ to denote the universal family of the Hilbert scheme of points. Applying relative Serre duality to the main result of \cite{Lehn} shows the tangent bundle of $\hns{n}{S}$ is given by $T_{\hns{n}{S}} = p_{2*} \hhom(\mathcal{I}_{\Zcaln},\Oc_{\Zcaln})$. The normal sequence for $\Zcaln$ gives a map:

\begin{center}
$\displaystyle p_1^* T_S \oplus p_2^* T_{\hns{n}{S}} \cong T_{S \times \hns{n}{S}}|_{\Zcaln} \xrightarrow{\beta} \big(\mathcal{I}_{\Zcaln}/\mathcal{I}_{\Zcaln}^2 \big)^\vee \cong \hhom(\mathcal{I}_{\Zcaln},\Oc_{\Zcaln})$.
\end{center}

\noindent Thus after pushing forward the first summand we get a map:

\begin{center}
$\displaystyle \alpha_n :\tnl \dfn p_{2*}(p_1^* T_S) \ra p_{2*}\hhom(\mathcal{I}_{\Zcaln},\Oc_{\Zcaln}) =T_{\hns{n}{S}}.$
\end{center}

To prove that $\alpha_n$ maps $\tnl$ isomorphically onto $\dern$ we first restrict to the open set $U \subset \hns{n}{S}$ parametrizing subschemes $\xi \subset S$ where $\xi$ contains at least $n-1$ distinct points. The complement of $U$ has codimension 2 so by reflexivity it is enough to prove the theorem on $U$. Moreover the open set

\begin{center}
$V \dfn p_2^{-1}U \subset \Zcaln$
\end{center}

\noindent is smooth so we are in a situation where we can apply Lemma 2.3.2. There is a map:

\begin{center}
\begin{tikzpicture}
  \node (tnl) {$p_2^*\tnl|_V$};
  \node (seq) [below] at (0,-1.2) {\hspace{1.17cm}$0 \ra T_{\Zcaln}|_V \ra p_2^*T_{\hns{n}{S}}|_V \oplus p_1^*T_S|_V \xrightarrow{\beta} \hhom(\mathcal{I}_{\Zcaln},\Oc_{\Zcaln})|_V$,};
  \draw[->] (tnl) to node[right] {$p_2^* \alpha_n|_V \oplus -\phi|_V$} (0,-1.4);
\end{tikzpicture}
\end{center}

\noindent where $\phi$ is the natural map coming from pulling back a pushforward. The composition:
\begin{center}
$\beta \circ (p_2^* \alpha_n|_V \oplus -\phi|_V)$
\end{center}

\noindent is identically zero. Therefore, the pullback of each local section of $\tnl|_U$ lies in $T_{\Zcaln}|_V$. It follows from Lemma 2.3.2 that $\tnl$ is contained in $\dern$. Now we can think of $\alpha_n$ as having codomain $\dern$. The map is an isomorphism of $\tnl$ and $\dern$ away from $B_n$ and they both have the same first Chern class. Therefore, $\alpha_n$ could only fail to be an isomorphism in codimension greater than 2. But both sheaves are reflexive, and any isomorphism between reflexive sheaves away from codimension 2 on a normal variety extends uniquely to an isomorphism on the whole variety.
\end{proof}

At a point $\br{\xi} \in \hns{n}{S}$ we have the isomorphisms $\tnl|_{\br{\xi}} \cong H^0(S,{T_S}|_{\xi})$ and $T|_{\hns{n}{S}} \cong \mathrm{Hom}(I_{\xi},\Oc_{\xi})$. From this perspective we describe the map $\alpha_n|_{\br{\xi}}$. If $\delta \in T_S|_{\br{\xi}}$ is derivation, then $\alpha_n|_{\br{\xi}}$ maps

\begin{center}
$\displaystyle \alpha_n|_{\br{\xi}}: \delta \mapsto \Big( \begin{array}{c}
    I_{\xi} \xrightarrow{\alpha_n|_{\br{\xi}}(\delta)} \Oc_{\xi} \\
    f \mapsto \delta(f)|_{\xi}
  \end{array} \Big)$.
\end{center}

From this description of $\alpha_n|_{\br{\xi}}$ we can easily compute the rank at any explicit point. For example, $\alpha_n$ is an isomorphism on the locus of $n$ distinct points, and the generic rank along $B_n$ is 2n-1. The degeneracy loci of $\alpha_n$ are defined to be:

\begin{center}
$\displaystyle \Omega_{r}(\alpha_n) \dfn \big\{ \br{\xi} \in \hns{n}{S} | \rk{\alpha_n|_{\br{\xi}}} \le r \big\}$.
\end{center}

\noindent We expect the irreducible components of $\Omega_{r}(\alpha_n)$ to give an interesting stratification of the Hilbert scheme of points, and we hope to return to study this stratification in future work.

\begin{subremark}
When $C$ is a smooth curve, the locus of nonreduced subschemes $B_n \subset \hns{n}{C}$ forms a divisor. When $C \cong \mathbb{A}^1$, this is the discriminant divisor in the space parametrizing degree $n$ polynomials in 1 variable. A similar proof shows the tautological bundle of the tangent bundle is the sheaf $\dern$, again showing $B_n$ is a free divisor.
\end{subremark}


\section{Perturbation of Polarization and Stability}

The goal of this section is to sketch a proof that stability of the tautological bundles with respect to the natural Chow divisors implies stability with respect to nearby ample divisors. We also prove in Proposition 3.7 that the pullback of a stable bundle to a product is stable with respect to a product polarization, a fact that we used in the proof of \hr{A}{Theorem A}. In proving stability with respect to nearby ample divisors we follow ideas appearing in ~\cite{Greb1} where compact moduli spaces of vector bundles are constructed using multipolarizations and more recently in ~\cite{Greb2} where foundational results for studying stability with respect to movable curves are established.

Throughout this section denote by X a normal complex projective variety of dimension $d$. Let $\gamma \in N_1(X)_{\mathbb{R}}$ be a real curve class and $\Ec$ be a torsion-free sheaf on X. For any sheaf $\Qc$ on $X$, we denote by $\sing{\Qc}$ the closed locus where $\Qc$ is not locally free.


\begin{definition}
The \textit{slope of $\Ec$ with respect to $\gamma$}, denoted by $\slgf{\gamma}{\Ec}$, is the real number:

\begin{center}
$\slgf{\gamma}{\Ec} \dfn \dfrac{c_1(\Ec) \cdot \gamma}{ \rk{\Ec}}$.
\end{center}
\end{definition}

\begin{remark}
Fixing an ample class $H \in N^1(X)_{\mathbb{R}}$ it is true that $\slhf{H}{\Ec} = \slgf{H^{d-1}}{\Ec}$. Nonetheless, to distinguish the concepts we use subscripts to denote slope with respect to an ample divisor and superscripts to denote slope with respect to a curve class.

\end{remark}


\begin{definition} We say $\Ec$ is \textit{slope (semi)stable with respect to $\gamma$} if for all torsion-free quotients of intermediate rank $\Ec \ra \Qc \ra 0$:

\begin{center}
$\slgf{\gamma}{\Ec} \underset{(\le)}{<} \slgf{\gamma}{\Qc}$.
\end{center}
\end{definition}

\noindent

A benefit of working with slope (semi)stability with respect to curves rather than divisors is that we can apply ideas of convexity.\label{cone}


\begin{lemma}\label{sum} If $\gamma, \delta$ are classes in $N_1(X)_{\mathbb{R}}$ such that $\Ec$ is semistable with respect to $\gamma$ and $\Ec$ is stable with respect to $\delta$ then $\Ec$ is stable with respect to $a \gamma + b \delta$ for $a,b > 0$. \qed
\end{lemma}

If $C \subset X$ is an irreducible curve we would like to relate the stability of $\Ec|_C $ and the stability of $\Ec$ with respect to the class of $C$. However if $\Qc$ is a coherent sheaf and $C$ meets $\sing{\Qc}$ it is possible that $c_1(\Qc|_C) \ne c_1(\Qc)|_C$. Thankfully we can say something if $C$ is not entirely contained in $\sing{\Qc}$.


\begin{proposition}\label{3.5} Let $\Ec \ra \Qc \ra 0$ be a torsion-free quotient which destabilizes $\Ec$ with respect to the curve class $\gamma$. Suppose $C \subset X$ is a smooth irreducible closed curve which represents $\gamma$, avoids $\sing{\Ec}$, and avoids the singularities of $X$. If $C$ is not contained in $\sing{Q}$ then $\Ec|_C$ is not stable on $C$.
\end{proposition}

\begin{proof} First, we can reduce to the surface case by choosing a normal surface $S \subset X$ containing $C$ such that $S$ is smooth along $C$, $S$ meets $\sing{\Qc}$ properly, and $S$ meets $\sing{\Ec}$ properly. This is possible because when the dimension of $X$ is greater than 3 a generic, high-degree hyperplane section containing $C$ is normal, smooth along $C$, and meets both $\sing{\Qc}$ and $\sing{\Ec}$ properly. Once such a surface is chosen

\begin{center}
$c_1(\Qc)|_S = c_1(\Qc|_S)= c_1(\Qc|_S / \mathrm{Tors}(\Qc|_S))$

$c_1(\Ec)|_S = c_1(\Ec|_S)= c_1(\Ec|_S / \mathrm{Tors}(\Ec|_S))$
\end{center}

\noindent because both $\sing{\Qc}\cap S$ and $\sing{\Ec}\cap S$ are zero-dimensional. Thus
\begin{center}
$\Ec|_S / \mathrm{Tors}(\Ec|_S) \ra \Qc|_S/\mathrm{Tors}(\Qc|_S) \ra 0$
\end{center}
is a torsion-free quotient on S which destabilizes $\Ec|_S / \mathrm{Tors}(\Ec|_S)$ with respect to the class of $C$. So we have reduced the Proposition to the case $X$ is a surface.

Let $X$ be a surface. It is enough to show $c_1(\Qc|_C) = c_1(\Qc)|_C$. The restriction $c_1(\Qc)|_C$ is computed via the derived pullback:

\begin{center}
$\displaystyle c_1(\Qc)|_C = \overset{\infty}{\underset{i = 0}\sum} (-1)^i c_1({\mathrm{Tor}_i}^{\Oc_X}(\Qc,\Oc_C))$,
\end{center}

\noindent where the ${\mathrm{Tor}_i}^{\Oc_X}(\Qc,\Oc_C)$ are thought of as modules on $C$. Further, $C$ is a Cartier divisor on $X$, so $\Oc_C$ has a two term locally free resolution. So the $\mathrm{Tor}_i^{\Oc_{X}}(\Qc,\Oc_C)$ vanish for $i>2$ and $\mathrm{Tor}_1^{\Oc_{X}}(\Qc,\Oc_C) = 0$ because $\Qc$ is torsion-free. Therefore

\begin{center}
$c_1(\Qc)|_C = c_1({\mathrm{Tor}_0}^{\Oc_X}(\Qc,\Oc_C)) = c_1(\Qc|_C)$.
\end{center}

\noindent So $\Ec|_C$ is not slope stable.
\end{proof}

An immediate corollary is the following coarse criterion for checking slope stability with respect to $\gamma$.


\begin{corollary}\label{3.6} Let $\pi: C_T \ra T$ be a family of smooth irreducible closed curves in $X$ with class $\gamma$. For $t \in T$ we write $C_t$ to denote $\pi^{-1}(t)$. Suppose $\Ec$ is a vector bundle on $X$ such that $\Ec|_{C_t}$ is stable for all $t \in T$. If the curves in $C_T$ are dense in $X$ then $\Ec$ is stable with respect to the curve class $\gamma$.
\end{corollary}

\begin{proof} Suppose for contradiction that $\Ec$ is unstable with respect to $\gamma$. Then there exists a torsion-free quotient $\Ec \ra \Qc \ra 0$ with $\slgf{\gamma}{\Qc} \le \slgf{\gamma}{\Ec}$. As $\Qc$ is torsion-free, $\sing{\Qc}$ has codimension $\ge 2$. The curves in $C_T$ are dense in $X$ so there is a $t\in T$ such that $C_t$ is not contained in $\sing{\Qc}$. Then \hr{1.6}{Proposition 1.6} guarantees that $\Ec|_{C_t}$ is not stable which contradicts our hypothesis.
\end{proof}

\hr{3.5}{Proposition 3.5} can be adjusted so that \hr{3.6}{Corollary 3.6} also holds if stability is replaced by semistability. As a consequence we prove the following basic result about slope stable vector bundles, which we used in the proof of \hr{A}{Theorem A}.


\begin{proposition}\label{3.7} Let $X$ and $Y$ be smooth projective varieties of dimension $d$ and $e$ respectively. Let $H_X$ be an ample divisor on $X$ (resp. $H_Y$ ample on $Y$) and let $p_1$ (resp. $p_2$) denote the projection from $X \times Y$ to $X$ (resp. $Y$). If $\Ec$ is a vector bundle on X which is slope stable with respect to $H_X$ then $p_1^*(\Ec)$ is slope stable on $X \times Y$ with respect to the ample divisor $p_1^*(H_X) + p_2^*(H_Y)$.
\end{proposition}

\begin{proof} By ~\cite[Theorem 4.3]{MehtaR} if $k \gg 0$ and $C$ is a general curve which is a complete intersection of divisors linearly equivalent to $kH_X$ then $\Ec|_C$ is stable. Let $F \subset |kH_X|^{d-1}$ be the open subset of the cartesian power of the complete linear series of $kH_X$ defined as
\[F \dfn \left\{ (H_1 , ... , H_{d-1}) \in |kH_X|^{d-1} \Big\vert
  \begin{array}{c}
    C = H_1 \cap ... \cap H_{d-1} \text{ is a smooth complete}\\
    \text{intersection curve and }\Ec|_C\text{ is stable}
  \end{array}
\right\} \subset |kH_X|^{d-1}.
\]
We write $C_F$ for the natural family of smooth curves in $X$ parametrized by $F$. Likewise the fiber product $C_F \times_F (F \times Y)$ is naturally a family of smooth curves in $X \times Y$ parametrized by $F \times Y$. The image of $C_F \times_F (F \times Y)$ in $X \times Y$ is dense, and for any $(f,y) \in F \times Y$ the restriction of $p_1^*(\Ec)$ to $C_{(f,y)}$ is stable. Therefore by \hr{1.7}{Corollary 1.7} $p_1^*(\Ec)$ is stable with respect to the numerical class of $C_{(f,y)}$ which we denote by $\gamma$.

For $l\gg0$ the divisor $l H_Y$ is very ample on $Y$ and a general complete intersection of divisors linearly equivalent to $l H_Y$ is smooth. Let $G \subset |lH_Y|^{e-1}$ be the open subset of the cartesian power of the complete linear series of $lH_Y$ defined as
\[G \dfn \left\{ (H_1 , ... , H_{e-1}) \in |lH_Y|^{e-1} \Big\vert
  \begin{array}{c}
    H_1 \cap ... \cap H_{e-1} \text{ is a smooth complete}\\
    \text{intersection curve}
  \end{array}
\right\} \subset |lH_Y|^{e-1}.
\]
As before there is a natural family $D_G$ of smooth curves in $Y$ parametrized by $G$. The fiber product $D_G \times_G (X \times G)$ is a family of smooth curves in $X \times Y$ parametrized by $X \times G$. For $(x,g) \in X \times G$ the restriction of $p_1^*(\Ec)$ to $D_{(x,g)}$ is a direct sum of trivial bundles thus the restriction is semistable. Therefore by applying \hr{1.7}{Corollary 1.7} in the semistable case, $p_1^*(\Ec)$ is semistable with respect to the curve class of $D_{(x,g)}$ which we write $\delta$.

Finally,

\begin{center}
$\displaystyle (p_1^*H_X + p_2^*H_Y)^{d+e-1} = {d+e-1 \choose e}\frac{(H_Y)^e}{k^{d-1}}\cdot\gamma + {d+e-1 \choose d}\frac{(H_X)^d}{l^{e-1}}\cdot\delta$.
\end{center}

\noindent Therefore by \hr{sum}{Lemma 1.5} $p_1^*(\Ec)$ is slope stable with respect to $p_1^*(H_X) + p_2^*(H_Y)$.
\end{proof}

This completes the proof of \hr{A}{Theorem A}. We now sketch a proof of the perturbation argument. The basic idea is an extension of the wall and chamber construction of \cite[Theorem 6.6]{Greb1} to the boundary of the positive cone of curves. This is possible when considering nef divisors which are also lef in the sense of \cite[Definition 2.1.3]{dCM1}.


\begin{proposition}\label{1.9} Let $H$ be a nef divisor and $A$ an ample $\mathbb{Q}$-divisor on $X$ a normal complex projective variety. Suppose $\Ec$ is a rank $r$ torsion-free sheaf on $X$ which is slope stable with respect to the class of $H^{d-1}$. Assume

\begin{center}
$- \cap H^{d-2}:N^1(X)_\mathbb{R} \ra N_1(X)_{\mathbb{R}}$

\hspace{2.5mm} $ \xi \mapsto \xi \cdot H^{d-2}$
\end{center}

\noindent is an isomorphism, then there is a nonempty open set in the ample cone abutting $H$ where $\Ec$ is stable.
\end{proposition}

This implies we can perturb our Chow polarization in the case of Hilbert schemes of points on surfaces and Chow divisors.


\begin{corollary}
If $\Ec$ is a vector bundle on $S$ a smooth projective surface which is stable with respect to $H$ an ample divisor, then $\enl{n}{\Ec}$ is stable with respect to an ample divisor near the Chow divisor $H_n$.
\end{corollary}

\begin{proof}[Proof of Corollary.]
By \cite[Theorem 2.3.1]{dCM1} we know $H_n$ is lef, so $\enl{n}{\Ec}$ and $H_n$ satisfy the conditions of Proposition 3.8. Therefore $\enl{n}{\Ec}$ is stable with respect to ample divisors close to $H_n$.
\end{proof}

\begin{proof}[Sketch of Proof of Proposition 3.8]
For $A$ any ample $\qq$-divisor on $X$ we have a maximally slope-destabilizing quotient $\Ec \ra \Qc_A$. Thus we can bound the negativity of $\slgf{A^{n-1}}{\Qc}-\slgf{A^{n-1}}{\Ec}$ for every torsion-free quotient $\Ec \ra \Qc$ of intermediate rank. On the other hand, we can give the bound:

\begin{center}
$\slgf{H^{n-1}}{\Qc}-\slgf{H^{n-1}}{\Ec} \ge \frac{1}{r(r-1)}$
\end{center}

\noindent because $H$ is a $\zz$-divisor. Combining these bounds and by linearity of slopes in curve classes we see that for small $\epsilon >0$, $\Ec$ is stable with respect to the curve class $H^{n-1}+\epsilon A^{n-1}$.

Here we have two closed convex cones, the set of nef $\mathbb{R}$-divisors and the set of curve classes $\gamma \in N_1(X)_{\mathbb{R}}$ where $\Ec$ is semistable with respect to $\gamma$, which we call the \textit{semistable cone}. By considering all ample $\qq$-divisors of the type $tH + (1-t)A$ for rational $t \in (0,1\rbrack$ and taking the limit as $t$ goes to $0$, one can show the derivative of the $(n-1)$st power map sends the tangent vectors pointing into the nef cone to the tangent vectors pointing into the semistable cone.

The derivative at $H$ of the ($n-1$)st power map from $N^1(X)_{\mathbb{R}}$ to $N_1(X)_{\mathbb{R}}$ is $(n-1)H^{n-2}$, which by assumption is nondegenerate. Therefore, it maps tangent vectors pointing into the interior of the ample cone to tangent vectors pointing into the interior of the semistable cone. As $\Ec$ is stable with respect to $H^{n-1}$, by Lemma 3.4 it is stable with respect to any class on the interior of the semistable cone. Therefore, there is a nonempty open set in the ample cone abutting $H$ where $\Ec$ is stable.
\end{proof}

\end{document}